\documentclass[12pt, reqno]{amsart}
\usepackage{amsmath, amsthm, amscd, amsfonts, amssymb, graphicx, color}
\usepackage[bookmarksnumbered, colorlinks, plainpages]{hyperref}

\input{mathrsfs.sty}

\newtheorem{theorem}{Theorem}[section]
\newtheorem{lemma}[theorem]{Lemma}
\newtheorem{proposition}[theorem]{Proposition}
\newtheorem{corollary}[theorem]{Corollary}
\theoremstyle{definition}

\theoremstyle{remark}
\newtheorem{remark}[theorem]{Remark}
\numberwithin{equation}{section}
\begin{document}

\title [Some inequalities of matrix power and Karcher means ]{Some inequalities of matrix power and Karcher means for positive linear maps}

\author[R. Lashkaripour, M. Hajmohamadi,  M. Bakherad,   ]{R. Lashkaripour$^1$, M. Hajmohamadi$^2$, M. Bakherad$^3$}

\address{$^1$$^{,2}$$^{,3}$ Department of Mathematics, Faculty of Mathematics, University of Sistan and Baluchestan, Zahedan, I.R.Iran.}

\email{$^{1}$lashkari@hamoon.usb.ac.ir}
\email{$^{2}$monire.hajmohamadi@yahoo.com}

\email{$^{3}$mojtaba.bakherad@yahoo.com; bakherad@member.ams.org}

\subjclass[2010]{47A64, 47A63, 47A60, 47A30.}

\keywords{Matrix power means, Karcher means,positive definite matrix, Positive linear mapping, unitarily invariant norm.}
%==================================================================================================================================
\begin{abstract}
In this paper, we generalize some matrix inequalities involving matrix power and Karcher means of positive definite matrices. Among other inequalities, it is shown that if
${\mathbb A}=(A_{1},...,A_{n})$ is a $n$-tuple of positive definite matrices such that  $0<m\leq A_{i}\leq M\, (i=1,\cdots,n)$ for some scalars $m< M$ and $\omega=(w_{1},\cdots,w_{n})$ is a weight vector with $w_{i}\geq0$ and $\sum_{i=1}^{n}w_{i}=1$, then
\begin{align*}
\Phi^{p}\Big(\sum_{i=1}^{n}w_{i}A_{i}\Big)\leq \alpha^{p}\Phi^{p}(P_{t}(\omega; {\mathbb A}))
\end{align*}
and
\begin{align*}
\Phi^{p}\Big(\sum_{i=1}^{n}w_{i}A_{i}\Big)\leq \alpha^{p}\Phi^{p}(\Lambda(\omega; {\mathbb A})),
\end{align*}
where $p>0$, $\alpha=\max\Big\{\frac{(M+m)^{2}}{4Mm}, \frac{(M+m)^{2}}{4^{\frac{2}{p}}Mm}\Big\}$,  $\Phi$  is a positive unital linear map and  $t\in [-1, 1]\backslash \{0\}$.
\end{abstract} \maketitle
%==================================================================================================================================
\section{Introduction and preliminaries}

\noindent Let $\mathcal{M}_n$ be the $C^*$-algebra of all
$n\times n$ complex matrices and  $\langle\,.\,,\,.\,\rangle$ be the standard scalar
product in $\mathbb{C}^n$ with the identity $I$. For Hermitian matrices $A, B\in
\mathcal{M}_n$, we write  $A\geq 0$ if $A$ is positive semidefinite,
$A>0$ if $A$ is positive definite, and $A\geq B$ if $A-B\geq0$. If $m, M$ be real scalars, then we mean $m\leq A\leq M$ that $mI\leq A\leq MI$.

The Gelfand map $f(t)\mapsto f(A)$ is an
isometrical $*$-isomorphism between the $C^*$-algebra
$C({\rm sp}(A))$ of continuous functions on the spectrum ${\rm sp}(A)$
of a Hermitian matrix $A$ and the $C^*$-algebra generated by $A$ and $I$. If $f, g\in C({\rm sp}(A))$, then
$f(t)\geq g(t)\,\,(t\in{\rm sp}(A))$ implies that $f(A)\geq g(A)$. A linear map $\Phi$ on $\mathcal{M}_n$ is positive if $\Phi(A)\geq0$ whenever $ A\geq0$. It is said to be unital if $\Phi(I)=I$. A norm $|||\cdot|||$ on $\mathcal{M}_n$ is said to be unitarily invariant norm if $|||UAV|||=|||A|||$, for all unitary matrices $U$ and $V$.\\
Let $A, B\in \mathcal{M}_n$ be two  positive definite and $t\in [0, 1]$. The operator t-weighted arithmetic, geometric, and harmonic means of $A, B$ are defined by $A\nabla_{t}B=(1-t)A+tB$, $A\sharp_{t}B=A^{\frac{1}{2}}(A^{-\frac{1}{2}}BA^{-\frac{1}{2}})^{t}A^{\frac{1}{2}}$ and $A!_{t}B=((1-t)A^{-1}+tB^{-1})^{-1}$ respectively, in which $A!_{t}B\leq A\sharp_{t}B\leq A\nabla_{t}B.$
In particular, for $t=\frac{1}{2}$ we get the operator arithmetic mean $\nabla$, the geometric mean $\sharp$ and the harmonic mean $!$. The AM-GM inequality reads
\begin{align}\label{32}
\frac{A+B}{2}\geq A\sharp B.
\end{align}
In \cite{pal}, Lim and Palfia have introduced matrix power means of positive definite matrices of some fixed dimension.
 If ${\mathbb A}=(A_{1},\cdots,A_{n})$ is a $n$-tuple of positive definite matrices $A_i\,\,(i=1,\cdots,n)$ and $\omega=(w_{1},\cdots,w_{n})$ is a positive probability weight vector where $w_{i}\geq0\,\,(i=1,\cdots,n)$ and $\sum_{i=1}^{n}w_{i}=1$, then the matrix power means  $P_{t}(\omega; {\mathbb A})$ is defined to be the unique positive definite solution of the non-linear equation:
\begin{align*}
X=\sum_{i=1}^{n}w_{i}(X\sharp_{t}A_{i}), \,\, t\in (0, 1]
\end{align*}
For $t\in[-1,0)$, it is defined by  $P_{t}(\omega; {\mathbb A})= P_{-t}(\omega; {\mathbb A^{-1}})^{-1}$, where ${\mathbb A}^{-1}=(A_{1}^{-1},\cdots,A_{n}^{-1})$.\\
 We denote $P_{1}(\omega; {\mathbb A})=\sum_{i=1}^{n}w_{i}A_{i}$ and $P_{-1}(\omega; {\mathbb A})=(\sum_{i=1}^{n}w_{i}A_{i}^{-1})^{-1}$, the weighted arithmetic and harmonic means of $A_{1},\cdots,A_{n}$, respectively.\\
There is one of important properties of matrix power means $P_{t}(\omega; {\mathbb A})$, that $P_{t}(\omega; {\mathbb A})$ interpolates between the weight harmonic and arithmetic means:
\begin{align}\label{11}
\left(\sum_{i=1}^{n}w_{i}A_{i}^{-1}\right)^{-1}\leq P_{t}(\omega; {\mathbb A})\leq \sum_{i=1}^{n}w_{i}A_{i}
\end{align}
for all $t\in [-1, 1]\backslash\{0\}$.\\
The Karcher means of $n$ positive probability vectors in ${\mathbb R^{n}}$ convexity spanned by the unit coordinate vectors, is defined as the unique positive definite solution of the equation:
\begin{align}\label{1}
\sum_{i=1}^{n}w_{i}\log\left(X^{\frac{1}{2}}A_{i}^{-1}X^{\frac{1}{2}}\right)=0.
\end{align}
The Karcher means denoted by $\Lambda(\omega; {\mathbb A})$, where it follows from \eqref{1} that $\Lambda(\omega; {\mathbb A}^{-1})^{-1}=\Lambda(\omega; {\mathbb A})$. It is well known that (see \cite{pal})
\begin{align}\label{14}
\lim_{t\rightarrow 0} P_{t}(\omega; {\mathbb A})=\Lambda(\omega; {\mathbb A})
\end{align}
and
\begin{align}\label{karcher}
\left(\sum_{i=1}^{n}w_{i}A_{i}^{-1}\right)^{-1}\leq \Lambda(\omega; {\mathbb A})\leq \sum_{i=1}^{n}w_{i}A_{i}.
\end{align}
For further information about the matrix power mean, Karcher mean and their properties, we refer the readers to \cite{pal, Law, Yam} and references therein.\\
It is well known that for the two positive  definite matrices $A, B$, if $A\geq B$, then
\begin{equation}\label{2}
  A^{p}\geq B^{p}\quad(0\leq p\leq1).
\end{equation}
In general \eqref{2} is not true for $p>1$.
Let $\Phi$ be a unital positive linear map. The following inequality is known as the Choi inequality see \cite{Ch,Fu}.% and the Kadison inequality respectively,
\begin{align}\label{10}
\Phi(A)^{-1}\leq \Phi(A^{-1}).
\end{align}
%\begin{align*}
%\Phi(A)^{2}\leq \Phi(A^{2}).
%\end{align*}
Ando \cite{And} proved that if $\Phi$ is a positive linear map, then for positive  definite matrices $A, B\in \mathfrak{B}(\mathcal{H})$ we have
\begin{equation}\label{an}
\Phi(A\sharp B)\leq \Phi(A)\sharp \Phi(B).
\end{equation}
A reverse of the Ando's inequality \eqref{an} is as follows: If $A, B\in \mathcal M_{n}$ and $0<m\leq A,B\leq M$, Then
\begin{align*}
\Phi(A)\sharp \Phi(B)\leq \frac{M+m}{2\sqrt{Mm}}\Phi(A\sharp B).
\end{align*}
By inequality \eqref{2} we get
\begin{align}\label{31}
(\Phi(A)\sharp \Phi(B))^{p}\leq \Big(\frac{M+m}{2\sqrt{Mm}}\Big)^{p}\Phi^{p}(A\sharp B),\, \, \, (0<p\leq 1).
\end{align}
Marshal and Olkin \cite{Mar} proved that a counterpart of Choi's inequality \eqref{10} as follows
\begin{align}\label{16}
\Phi(A^{-1})\leq\frac{(M+m)^{2}}{4Mm}\Phi(A)^{-1}
\end{align}
for positive definite $A$ with $0<m\leq A\leq M$. In addition  Lin \cite{lin} and Fu \cite{fu} improved inequality \eqref{16} for $p\geq2$.\\
The matrix power means satisfy the following inequality: for each $t\in (0, 1]$
\begin{align}\label{15}
\Phi (P_{t}(\omega; {\mathbb A}))\leq P_{t}(\omega; \Phi({\mathbb A})),
\end{align}
where ${\mathbb A}=(A_{1},\cdots,A_{n})$ be a $n$-tuple of positive definite matrices and $\Phi({\mathbb A})=(\Phi(A_{1}),\\\cdots
,\Phi(A_{n}))$.\\
Dehghani et al. \cite{deh} established counterparts of \eqref{15} involving matrix power means as following:
\begin{align*}
P_{t}^{2}(\omega; \Phi({\mathbb A}))\leq\Big(\frac{(m+M)^{2}}{4mM}\Big)^{2}\Phi^{2}(P_{t}(\omega; {\mathbb A}))
\end{align*}
for all $t\in [-1, 1]\backslash\{0\}$ and $0<m\leq A_{i}\leq M$.\\
Using inequality \eqref{2} we get
\begin{align}\label{3}
P_{t}^{p}(\omega; \Phi({\mathbb A}))\leq\Big(\frac{(m+M)^{2}}{4mM}\Big)^{p}\Phi^{p}(P_{t}(\omega; {\mathbb A})),\, \, \, (0<p\leq 2)
\end{align}
It is interesting to ask whenever the inequality \eqref{3} is true for $p\geq2$. This is the first motivation of this paper. moreover, we improve inequality \eqref{31} for $p\geq2$. We also obtain some reverses of \eqref{11}. In the last section, we establish several refinements of obtained inequalities.

%=================================================================================================================================
\section{Main results}
\bigskip To prove our first result, we need the following lemmas.
\begin{lemma}\cite{Bha,Ando,Ba,Fuji}\label{6}
Let $A, B\in\mathcal{M}_n$ be positive  definite matrices  and $\alpha>0$. Then \\
$\rm{(i)}\,\,||AB||\leq\frac{1}{4}||A+B||^{2}.$\\
$\rm{(ii)}\,\,||A^{\alpha}+B^{\alpha}||\leq ||(A+B)^{\alpha}||.$\\
$\rm{(iii)}\,\, A\leq\alpha B$ if and only if $||A^{\frac{1}{2}}B^{-\frac{1}{2}}||\leq \alpha^{\frac{1}{2}}.$\\
$\rm{(iv)}\,$ If $0\leq A\leq B$ and $0<m\leq A\leq M$, then $A^{2}\leq \frac{(M+m)^{2}}{4Mm}B^{2}.$
\end{lemma}
\begin{lemma}\cite{Ho}\label{19}
Let $A\in\mathcal{M}_n$ be positive definite. Then $A\leq tI$ if and only if $\|A\|\leq t$ if and only if $\left[\begin{array}{cc}
 tI&A\\
 A^{*}&tI
 \end{array}\right]$ is positive.
 \end{lemma}
\begin{theorem}\label{p}
Let ${\mathbb A}=(A_{1},\cdots,A_{n})$ be a $n$-tuple of positive definite matrices with $0<m\leq A_{i}\leq M,\ (i=1,\cdots,n)$ for some scalars $m< M$ and $\omega=(w_{1},\cdots,w_{n})$ a weight vector. If $\Phi$ is a unital positive linear map, then
\begin{align}\label{4}
P_{t}^{p}(\omega; \Phi({\mathbb A}))\leq\Big(\frac{(m+M)^{2}}{4^{\frac{2}{p}}mM}\Big)^{p}\Phi^{p}(P_{t}(\omega; {\mathbb A}))
\end{align}
for every $p\geq2$ and $t\in [-1, 1]\backslash\{0\}$.
\begin{proof}
By Lemma \ref{6}(iii), inequality \eqref{4} is equivalent to
\begin{align}\label{8}
\Big\|P_{t}^{\frac{p}{2}}(\omega; \Phi({\mathbb A}))\Phi^{-\frac{p}{2}}(P_{t}(\omega; {\mathbb A}))\Big\|\leq \frac{(m+M)^{p}}{4M^{\frac{p}{2}}m^{\frac{p}{2}}}.
\end{align}
Hence, it is enough to prove inequality \eqref{8}. So
\begin{align}\label{haj1}
M^{\frac{p}{2}}m^{\frac{p}{2}}\Big\|P_{t}^{\frac{p}{2}}(\omega; \Phi({\mathbb A}))\Phi^{-\frac{p}{2}}(P_{t}(\omega; {\mathbb A}))\Big\|\nonumber&=
\Big\|P_{t}^{\frac{p}{2}}(\omega; \Phi({\mathbb A}))M^{\frac{p}{2}}m^{\frac{p}{2}}\Phi^{-\frac{p}{2}}(P_{t}(\omega; {\mathbb A}))\Big\|\nonumber\\&
\leq \frac{1}{4}\Big\|P_{t}^{\frac{p}{2}}(\omega; \Phi({\mathbb A}))+M^{\frac{p}{2}}m^{\frac{p}{2}}\Phi^{-\frac{p}{2}}(P_{t}(\omega; {\mathbb A}))\Big\|^{2}\nonumber\\&
   \qquad   \qquad   \qquad   \qquad   \qquad (\textrm{by Lemma \ref{6}(i)})\nonumber\\&
\leq \frac{1}{4}\Big\|(P_{t}(\omega; \Phi({\mathbb A}))+Mm\Phi^{-1}(P_{t}(\omega; {\mathbb A})))^{\frac{p}{2}}\Big\|^{2}\nonumber\\&
  \qquad   \qquad   \qquad  \qquad   \qquad (\textrm{by Lemma \ref{6}(ii)})\nonumber\\&
=\frac{1}{4}\|(P_{t}(\omega; \Phi({\mathbb A}))+Mm\Phi^{-1}(P_{t}(\omega; {\mathbb A})))\|^{p}\nonumber\\&
\leq \frac{1}{4}\left\|\sum_{i=1}^{n}w_{i}\Phi(A_{i})+Mm\Phi (P_{t}(\omega; {\mathbb A})^{-1})\right\|^{p}\nonumber\\&
  \qquad   \qquad   \qquad  \qquad   \qquad   (\textrm{by \eqref{10}})\nonumber\\&
\leq \frac{1}{4}\left\|\sum_{i=1}^{n}w_{i}\Phi(A_{i})+Mm\Phi (\sum_{i=1}^{n}w_{i}A_{i}^{-1})\right\|^{p}\nonumber\\&
  \qquad   \qquad   \qquad  \qquad    \qquad (\textrm{by \eqref{11}})\nonumber\\&
=\frac{1}{4}\left\|\sum_{i=1}^{n}w_{i}\Big(\Phi(A_{i})+Mm\Phi(A_{i}^{-1})\Big)\right\|^{p}.
\end{align}
It follows from $0<m\leq A_i\leq M$ that $(M-A_i)(m-A_i)A_i^{-1}\leq0\,\,(i=1, 2,\cdots, n)$.
Hence
\begin{align}\label{haj2}
Mm\Phi(A_i^{-1})+\Phi(A_i)\leq M+m\qquad(i=1, 2,\cdots, n)
\end{align}
Using inequalities \eqref{haj1} and \eqref{haj2} we get
\begin{align*}
||P_{t}^{\frac{p}{2}}(\omega; \Phi({\mathbb A}))\Phi^{-\frac{p}{2}}(P_{t}(\omega; {\mathbb A}))||\leq \frac{(m+M)^{p}}{4M^{\frac{p}{2}}m^{\frac{p}{2}}}.
\end{align*}
Thus, this completes the proof.
\end{proof}
\end{theorem}
In the following result we state that inequality \eqref{31} is valid for any $p\geq2$.
\begin{corollary}\label{30}
Let $A, B\in \mathcal M(\mathbb C)$ be positive definite matrices such that $0<m\leq A,B\leq M$ for some scalars $m< M$ and $\alpha\in[0,1]$. Then
\begin{align*}
(\Phi(A)\sharp_{\alpha}\Phi(B))^{p}\leq \Big(\frac{(m+M)^{2}}{4^{\frac{2}{p}}mM}\Big)^{p}\Phi^{p}(A\sharp_{\alpha} B),
\end{align*}
for any $p\geq2$ and unital positive linear map $\Phi$.
\end{corollary}
\begin{proof}
Using this fact $P_{t}(1-\alpha,\alpha;A,B)=A\sharp_{\alpha}B,\ (\alpha\in [0,1])$ and $n=2, w_{1}=1-\alpha$ and $w_{2}=\alpha$ in inequality\eqref{4}, we get the desired result.
\end{proof}

\begin{corollary}
Let ${\mathbb A}=(A_{1},\cdots,A_{n})$ be a n-tuple of positive definite matrices with $0<m\leq A_{i}\leq M,\ (i=1,\cdots,n)$ for some scalars $m< M$ and $\omega=(w_{1},\cdots,w_{n})$ a weight vector. If $\Phi$ is a unital positive linear map, then
\begin{align*}
\Lambda^{p}(\omega; \Phi({\mathbb A}))\leq\Big(\frac{(m+M)^{2}}{4^{\frac{2}{p}}mM}\Big)^{p}\Phi^{p}(\Lambda(\omega; {\mathbb A}))
\end{align*}
for every $p\geq2$ and $t\in [-1, 1]\backslash\{0\}$.
\begin{proof}
The proof follows from Theorem \ref{p} and relation \eqref{14}.
\end{proof}
\end{corollary}

\begin{theorem}
Let ${\mathbb A}=(A_{1},\cdots,A_{n})$ be a n-tuple of positive definite matrices such that $0<m\leq A_{i}\leq M,\ (i=1,\cdots,n)$ for some scalars $m< M$ and $\omega=(w_{1},\cdots,w_{n})$ a weight vector. Then
\begin{align}\label{12}
\sum_{i=1}^{n}w_{i}A_{i}\leq \frac{(M+m)^{2}}{4Mm}P_{t}(\omega; {\mathbb A}),
\end{align}
where $t\in [-1, 1]\backslash\{0\}$.
\end{theorem}
\begin{proof}
If we put $\Phi(A)=\sum_{i=1}^{n}w_{i}A_{i}$, then for $t\in (0, 1]$ we have
\begin{align*}
\sum_{i=1}^{n}w_{i}A_{i}=\Phi(A)&\leq\frac{(M+m)^{2}}{4Mm}\Big(\sum_{i=1}^{n}w_{i}A_{i}^{-1}\Big)^{-1}   \qquad   \qquad (\textrm{by \eqref{16}})\\&
\leq \frac{(M+m)^{2}}{4Mm} P_{t}(\omega; {\mathbb A}) \qquad \qquad    \qquad (\textrm{by \eqref{11}}).
\end{align*}
Therefore
\begin{align*}
\sum_{i=1}^{n}w_{i}A_{i}\leq \frac{(M+m)^{2}}{4Mm}P_{t}(\omega; {\mathbb A}).
\end{align*}
Inequality \eqref{12} follows from a similar fashion for $t\in [-1, 0)$.
\end{proof}
\begin{remark}
As special case for $\mathbb A=(A,B)$ and $\omega=(w_{1},w_{2})$ with $w_{1}=w_{2}=\frac{1}{2}$, we have the following inequality:
\begin{align*}
\frac{A+B}{2}\leq\frac{(M+m)^{2}}{4Mm}(A\sharp B),
\end{align*}
which is counterpart of AM-GM inequality \eqref{32}.
\end{remark}
\begin{corollary}\label{AM}
Let ${\mathbb A}=(A_{1},\cdots,A_{n})$ be a n-tuple of positive definite matrices with $0<m\leq A_{i}\leq M,\ (i=1,\cdots,n)$ for some scalars $m\leq M$ and $\omega=(w_{1},\cdots,w_{n})$ a weight vector. Then
\begin{align}\label{kar}
\sum_{i=1}^{n}w_{i}A_{i}\leq \frac{(M+m)^{2}}{4Mm}\Lambda(\omega; {\mathbb A}),
\end{align}
where  $t\in [-1, 1]\backslash\{0\}$.
\end{corollary}
\begin{remark}
Inequalities \eqref{12} and \eqref{kar} can be regarded as a counterpart of inequalities \eqref{11} and \eqref{karcher}, respectively. By inequalities \eqref{12} and \eqref{15}, we can obtain the following operator inequality
\begin{align}\label{13}
\Phi\Big(\sum_{i=1}^{n}w_{i}A_{i}\Big)&\leq\frac{(M+m)^{2}}{4Mm}\Phi(P_{t}(\omega; {\mathbb A}))\nonumber\\&
\leq\frac{(M+m)^{2}}{4Mm}P_{t}(\omega; \Phi({\mathbb A})).
\end{align}
Now, by applying inequality \eqref{2} we get
\begin{align}\label{26}
\Phi^{p}\Big(\sum_{i=1}^{n}w_{i}A_{i}\Big)\leq\Big(\frac{(M+m)^{2}}{4Mm}\Big)^{p}P_{t}^{p}(\omega; \Phi({\mathbb A}))
\end{align}
for $0<p\leq1$.
\end{remark}
In the next theorem, we show that inequality \eqref{26} is valid for $p>1$.
\begin{theorem}\label{th}
Let ${\mathbb A}=(A_{1},\cdots,A_{n})$ be a $n$-tuple of positive definite matrices with $0<m\leq A_{i}\leq M,\ (i=1,\cdots,n)$ for some scalars $m< M$ and $\omega=(w_{1},\cdots,w_{n})$ a weight vector. Then
\begin{align}\label{17}
\Phi^{p}\Big(\sum_{i=1}^{n}w_{i}A_{i}\Big)\leq \alpha^{p}\Phi^{p}(P_{t}(\omega; {\mathbb A})),
\end{align}
where  $t\in [-1, 1]\backslash\{0\}$, $p>1$ and $\alpha=\max\Big\{\frac{(M+m)^{2}}{4Mm}, \frac{(M+m)^{2}}{4^{\frac{2}{p}}Mm}\Big\}$.
%$(ii)$ For $p>2$,
%\begin{align}\label{21}
%\Phi^{p}\Big(\sum_{i=1}^{n}w_{i}A_{i}\Big)\leq \Big(\frac{(M+m)^{2}}{4^{\frac{2}{p}}Mm}\Big)^{p}\Phi^{p}(P_{t}(\omega; {\mathbb A})).
%\end{align}
\end{theorem}
\begin{proof}
 First we show inequality \eqref{17} for $p=2$. We have
\begin{align*}
Mm\left\|\Phi\Big(\sum_{i=1}^{n}w_{i}A_{i}\Big)\Phi^{-1}(P_{t}(\omega; {\mathbb A}))\right\|&= \left\|\Phi\Big(\sum_{i=1}^{n}w_{i}A_{i}\Big)Mm\Phi^{-1}(P_{t}(\omega; {\mathbb A}))\right\|
\\&\leq\frac{1}{4}\left\|\Phi\Big(\sum_{i=1}^{n}w_{i}A_{i}\Big)+Mm\Phi^{-1}(P_{t}(\omega; {\mathbb A}))\right\|^{2}\\&
\qquad\qquad\qquad    \qquad (\textrm{by Lemma \ref{6}})\\&
\leq\frac{1}{4}\left\|\Phi\Big(\sum_{i=1}^{n}w_{i}A_{i}\Big)+Mm\Phi(\sum_{i=1}^{n}w_{i}A_{i}^{-1})\right\|^2\\&
\leq\frac{1}{4}(M+m)^2,
\end{align*}
 whence
 \begin{align*}
\left\|\Phi\Big(\sum_{i=1}^{n}w_{i}A_{i}\Big)\Phi^{-1}(P_{t}(\omega; {\mathbb A}))\right\|\leq \frac{(M+m)^{2}}{4Mm}.
\end{align*}
Hence
\begin{align*}
\Phi^{2}\Big(\sum_{i=1}^{n}w_{i}A_{i}\Big)\leq \Big(\frac{(M+m)^{2}}{4Mm}\Big)^{2}\Phi^{2}(P_{t}(\omega; {\mathbb A})).
\end{align*}
Therefore
\begin{align}\label{27}
\Phi^{p}\Big(\sum_{i=1}^{n}w_{i}A_{i}\Big)\leq \Big(\frac{(M+m)^{2}}{4Mm}\Big)^{p}\Phi^{p}(P_{t}(\omega; {\mathbb A})),\,\ (0\leq p\leq2)
\end{align}
Now, we prove  inequality \eqref{17} for $p>2$. In this case we have
\begin{align*}
\Big\|\Phi^{\frac{p}{2}}\Big(\sum_{i=1}^{n}w_{i}A_{i}\Big)M^{\frac{p}{2}}m^{\frac{p}{2}}&\Phi^{-\frac{p}{2}}(P_{t}(\omega; {\mathbb A})) \Big\|\\&\leq\frac{1}{4}\Big\|\Phi^{\frac{p}{2}}\Big(\sum_{i=1}^{n}w_{i}A_{i}\Big)+M^{\frac{p}{2}}m^{\frac{p}{2}}
\Phi^{-\frac{p}{2}}(P_{t}(\omega; {\mathbb A}))\Big\|^{2}\\&
 \qquad  \qquad   \qquad (\textrm{by Lemma \ref{6}(i)})\\&
\leq\frac{1}{4}\Big\|\Big(\Phi\Big(\sum_{i=1}^{n}w_{i}A_{i}\Big)+Mm\Phi^{-1}(P_{t}(\omega; {\mathbb A}))\Big)^{\frac{p}{2}}\Big\|^{2} \\&
\qquad  \qquad   \qquad (\textrm{by Lemma \ref{6}(ii)})\\&
=\frac{1}{4}\Big\|\Phi\Big(\sum_{i=1}^{n}w_{i}A_{i}\Big)+Mm\Phi^{-1}(P_{t}(\omega; {\mathbb A}))\Big\|^{p}\\&
\leq\frac{(M+m)^{p}}{4}.
\end{align*}
Hence
\begin{align*}
\left\|\Phi^{\frac{p}{2}}\Big(\sum_{i=1}^{n}w_{i}A_{i}\Big)\Phi^{-\frac{p}{2}}(P_{t}(\omega; {\mathbb A})) \right\|\leq \frac{1}{4}\Big(\frac{(M+m)^{p}}{M^{\frac{p}{2}}m^{\frac{p}{2}}}\Big).
\end{align*}
Thus
\begin{align}\label{28}
\Phi^{p}\Big(\sum_{i=1}^{n}w_{i}A_{i}\Big)\leq\Big(\frac{(M+m)^{2}}{4^{\frac{2}{p}}Mm}\Big)^{p}\Phi^{p}(P_{t}(\omega; {\mathbb A})).
\end{align}
Now, if we take $\alpha=\max\Big\{\frac{(M+m)^{2}}{4Mm}, \frac{(M+m)^{2}}{4^{\frac{2}{p}}Mm}\Big\}$, then by  \eqref{27} and \eqref{28} we get the desired result.
\end{proof}
\begin{remark}
By letting ${\mathbb A}=(A,B)$ and $\omega=(w_{1},w_{2})$ with $w_{1}=w_{2}=\frac{1}{2}$ in Theorem \ref{th}, the following inequalities are hold:
\begin{align*}
\Phi^{p}\Big( \frac{A+B}{2}\Big)\leq\alpha^{p}\Phi^{p}(A\sharp B),
\end{align*}
%in particular
%\begin{align*}
%\Phi^{2}\Big( \frac{A+B}{2}\Big)\leq\alpha^{2}\Phi^{2}(A\sharp B),
%\end{align*}
Which appeared in \cite[Theorem 4]{Fu}. where $\alpha=\max\Big\{\frac{(M+m)^{2}}{4Mm}, \frac{(M+m)^{2}}{4^{\frac{2}{p}}Mm}\Big\}$.
\end{remark}
\begin{corollary}
Let ${\mathbb A}=(A_{1},\cdots,A_{n})$ be a $n$-tuple of positive definite matrices with $0<m\leq A_{i}\leq M,\ (i=1,\cdots,n)$ for some scalars $m\leq M$ and $\omega=(w_{1},\cdots,w_{n})$ a weight vector, and let $t\in [-1, 1]\backslash\{0\}$. Then
\begin{align*}
\Phi^{p}\Big(\sum_{i=1}^{n}w_{i}A_{i}\Big)\leq \alpha^{p}\Phi^{p}\Lambda(\omega; {\mathbb A}),
\end{align*}
where $p\geq1$ and $\alpha=\max\Big\{\frac{(M+m)^{2}}{4Mm}, \frac{(M+m)^{2}}{4^{\frac{2}{p}}Mm}\Big\}$.
\end{corollary}

In the next result we extend  inequalities \eqref{4} and \eqref{17} to the follwing form.
 \begin{theorem}
 Let ${\mathbb A}=(A_{1},\cdots,A_{n})$ be a $n$-tuple of positive definite matrices with $0<m\leq A_{i}\leq M,\ (i=1,\cdots,n)$ for some scalars $m\leq M$ and $\omega=(w_{1},\cdots,w_{n})$ a weight vector, let $t\in [-1, 1]\backslash\{0\}$ and $\Phi$ be a positive unital linear map. Then
 \begin{align*}
 P_{t}^{p}(\omega; {\mathbb A})\Phi^{-p}(P_{t}(\omega; {\mathbb A}))+\Phi^{-p}(P_{t}(\omega; {\mathbb A}))P_{t}^{p}(\omega; {\mathbb A})\leq 2\alpha^{p}
 \end{align*}
 and
 \begin{align}\label{24}
 \Phi^{p}\Big(\sum_{i=1}^{n}w_{i}A_{i}\Big)\Phi^{-p}(P_{t}(\omega; {\mathbb A}))+\Phi^{-p}(P_{t}(\omega; {\mathbb A}))\Phi^{p}\Big(\sum_{i=1}^{n}w_{i}A_{i}\Big)\leq 2\alpha^{p},
 \end{align}
 where  $p>0$ and $\alpha=\max\Big\{\frac{(m+M)^{2}}{4mM}, \frac{(m+M)^{2}}{4^{\frac{1}{p}}mM}\Big\}.$
 \end{theorem}
 \begin{proof}
 By inequality \eqref{3} and Lemma \ref{6}(iii) for $0<p\leq 1$ we have
 \begin{align*}
 ||P_{t}^{p}(\omega; {\mathbb A})\Phi^{-p}(P_{t}(\omega; {\mathbb A}))||\leq\Big(\frac{(m+M)^{2}}{4mM}\Big)^{p}.
 \end{align*}
 We put $\alpha=\frac{(m+M)^{2}}{4mM}$. Using Lemma \ref{19} we get
  \begin{align*}
 \left[\begin{array}{cc}
 \alpha^{p}I&P_{t}^{p}(\omega; {\mathbb A})\Phi^{-p}(P_{t}(\omega; {\mathbb A}))\\
 \Phi^{-p}(P_{t}(\omega; {\mathbb A}))P_{t}^{p}(\omega; {\mathbb A})&\alpha^{p}I
 \end{array}\right]
 \end{align*}
 and
 \begin{align*}
 \left[\begin{array}{cc}
 \alpha^{p}I&\Phi^{-p}(P_{t}(\omega; {\mathbb A}))P_{t}^{p}(\omega; {\mathbb A})\\
 P_{t}^{p}(\omega; {\mathbb A})\Phi^{-p}(P_{t}(\omega; {\mathbb A}))\\&\alpha^{p}I
 \end{array}\right]
 \end{align*}
 are positive. Hence
 {\tiny\begin{align*}
 \left[\begin{array}{cc}
 2\alpha^{p}I&P_{t}^{p}(\omega; {\mathbb A})\Phi^{-p}(P_{t}(\omega; {\mathbb A}))+\Phi^{-p}(P_{t}(\omega; {\mathbb A}))P_{t}^{p}(\omega; {\mathbb A})\\
 \Phi^{-p}(P_{t}(\omega; {\mathbb A}))P_{t}^{p}(\omega; {\mathbb A})+P_{t}^{p}(\omega; {\mathbb A})\Phi^{-p}(P_{t}(\omega; {\mathbb A}))&2\alpha^{p}I
 \end{array}\right]
 \end{align*}}
is positive. Using Lemma \ref{19} we get
 \begin{align*}
 P_{t}^{p}(\omega; {\mathbb A})\Phi^{-p}(P_{t}(\omega; {\mathbb A}))+\Phi^{-p}(P_{t}(\omega; {\mathbb A}))P_{t}^{p}(\omega; {\mathbb A})\leq 2\alpha^{p}.
 \end{align*}
 For $p>1$, using inequality \eqref{4} with the same argument, we get the desired inequality.\\
 Inequality \eqref{24} is proved by using Theorem \ref{th} and a similar method.
 \end{proof}
 \begin{corollary}
 Let ${\mathbb A}=(A_{1},\cdots,A_{n})$ be a $n$-tuple of positive definite matrices with $0<m\leq A_{i}\leq M,\ (i=1,\cdots,n)$ for some scalars $m\leq M$ and $\omega=(w_{1},\cdots,w_{n})$ a weight vector, $\Phi$ be a positive unital linear map. Then
 \begin{align*}
\Lambda^{p}(\omega; {\mathbb A})\Phi^{-p}(\Lambda(\omega; {\mathbb A}))+\Phi^{-p}(\Lambda(\omega; {\mathbb A}))\Lambda^{p}(\omega; {\mathbb A})\leq 2\alpha^{p}
 \end{align*}
 and
 \begin{align}\label{24}
 \Phi^{p}\Big(\sum_{i=1}^{n}w_{i}A_{i}\Big)\Phi^{-p}(\Lambda(\omega; {\mathbb A}))+\Phi^{-p}(\Lambda(\omega; {\mathbb A}))\Phi^{p}\Big(\sum_{i=1}^{n}w_{i}A_{i}\Big)\leq 2\alpha^{p},
 \end{align}
 where  $p>0$ and $\alpha=\max\Big\{\frac{(m+M)^{2}}{4mM}, \frac{(m+M)^{2}}{4^{\frac{1}{p}}mM}\Big\}.$
 \end{corollary}

In the next result, we would like to obtain unitary invariant norm inequality involving matrix power means.
\begin{proposition}\label{18}
Let ${\mathbb A}=(A_{1},\cdots,A_{n})$ be a $n$-tuple of positive definite matrices with $0<m\leq A_{i}\leq M,\ (i=1,\cdots,n)$ for some scalars $m\leq M$ and $\omega=(w_{1},\cdots,w_{n})$ a weight vector, and let $|||\cdot|||$ be an unitary invariant norm. Then for $t\in (0, 1]$\\
$|||P_{t}(\omega; {\mathbb A})|||\leq\sum_{i=1}^{n}w_{i}|||A_{i}|||$ \,\,\,\,\,\, and \,\,\,\,\ $|||P_{-t}(\omega; {\mathbb A})|||\geq\Big(\sum_{i=1}^{n}w_{i}|||A_{i}^{-1}|||\Big)^{-1}$.
\end{proposition}
\begin{proof}
Let $X=P_{t}(\omega; {\mathbb A})$. Then
\begin{align*}
|||X|||=|||P_{t}(\omega; {\mathbb A})|||&\leq\sum_{i=1}^{n}w_{i}|||X\sharp_{t}A_{i}|||\\&
\leq \sum_{i=1}^{n}w_{i}|||(1-t)X+tA_{i}|||\\&
\leq |||(1-t)X|||\sum_{i=1}^{n}w_{i}+t\sum_{i=1}^{n}w_{i}|||A_{i}|||,
\end{align*}
 which implies that $|||P_{t}(\omega; {\mathbb A})|||\leq\sum_{i=1}^{n}w_{i}|||A_{i}|||$. For second inequality, it follows from $|||A^{-1}|||\geq|||A|||^{-1}$ for any $A>0$ that
\begin{align*}
|||P_{-t}(\omega; {\mathbb A})|||=|||P_{t}(\omega; {\mathbb A}^{-1})^{-1}|||\geq|||P_{t}(\omega; {\mathbb A}^{-1})|||^{-1}\geq\Big(\sum_{i=1}^{n}w_{i}|||A_{i}^{-1}|||\Big)^{-1}.
\end{align*}
\end{proof}

%==================================================================================================================================
\section{ Some refinements}
In this section, we give a refinement of inequality \eqref{17}. This inequality can be refined by a similar method that known in \cite{zh}.
\begin{theorem}\label{ref}
Let ${\mathbb A}=(A_{1},\cdots,A_{n})$ be a $n$-tuple of positive definite matrices with $0<m\leq A_{i}\leq M,\ (i=1,\cdots,n)$ for some scalars $m\leq M$ and $\omega=(w_{1},\cdots,w_{n})$ a weight vector, and let $t\in [-1, 1]\backslash\{0\}$. Then for every positive unital linear map $\Phi$
\begin{align}\label{23}
\Phi^{2p}\Big(\sum_{i=1}^{n}w_{i}A_{i}\Big)\leq\frac{(K(M^{2}+m^{2}))^{2p}}{16M^{2p}m^{2p}}\Phi^{2p}(P_{t}(\omega; {\mathbb A})),
\end{align}
where $p\geq2$ and $K=\frac{(M+m)^{2}}{4mM}$.
\end{theorem}
\begin{proof}
%In inequality \eqref{17} as $p=2$ we have
%\begin{align*}
%\Phi^{-2}(P_{t}(\omega; {\mathbb A}))\leq K^{2}\Phi^{-2}\Big(\sum_{i=1}^{n}w_{i}A_{i}\Big).
%\end{align*}
For $p\geq2$, we have
\begin{align*}
\Big\|\Phi^{p}\Big(\sum_{i=1}^{n}w_{i}A_{i}\Big)&M^{p}m^{p}\Phi^{-p}(P_{t}(\omega; {\mathbb A}))\Big\|\\&\leq
\frac{1}{4}\Big\|K^{\frac{p}{2}}\Phi^{p}\Big(\sum_{i=1}^{n}w_{i}A_{i}\Big)+(\frac{M^{2}m^{2}}{K})^{\frac{p}{2}}\Phi^{-p}(P_{t}(\omega; {\mathbb A}))\Big\|^{2}\\&
\qquad  \qquad   \qquad \qquad \qquad \qquad  (\textrm{by Lemma \ref{6}(i)})\\&
\leq\frac{1}{4}\Big\|\Big(K\Phi^{2}\Big(\sum_{i=1}^{n}w_{i}A_{i}\Big)+\frac{M^{2}m^{2}}{K}\Phi^{-2}(P_{t}(\omega; {\mathbb A}))\Big)^{\frac{p}{2}}\Big\|^{2}\\&
\qquad  \qquad   \qquad \qquad \qquad \qquad (\textrm{by Lemma \ref{6}(ii)})\\&
=\frac{1}{4}\Big\|\Big(K\Phi^{2}\Big(\sum_{i=1}^{n}w_{i}A_{i}\Big)+\frac{M^{2}m^{2}}{K}\Phi^{-2}(P_{t}(\omega; {\mathbb A}))\Big)\Big\|^{p}\\&
\leq
\frac{1}{4}\Big\|\Big(K\Phi^{2}\Big(\sum_{i=1}^{n}w_{i}A_{i}\Big)+\frac{M^{2}m^{2}}{K}\Phi^{2}(P_{t}(\omega; {\mathbb A})^{-1})\Big)\Big\|^{p}\\&
\qquad  \qquad   \qquad \qquad \qquad \qquad (\textrm{by \eqref{10}})\\&
\leq \frac{1}{4}\Big\|K\Phi^{2}\Big(\sum_{i=1}^{n}w_{i}A_{i}\Big)+KM^{2}m^{2}\Phi^{2}
\Big(\sum_{i=1}^{n}w_{i}A_{i}^{-1}\Big)\Big\|^{p}\\&
\qquad  \qquad   \qquad \qquad \qquad \qquad (\textrm{by Lemma \ref{6}(iv)})\\&
=\frac{1}{4}(K(M^{2}+m^{2}))^{p}\qquad  \qquad   \qquad (\textrm{by \cite[4.7]{mlin}}).
\end{align*}
Hence
\begin{align}\label{22}
\Big\|\Phi^{p}\Big(\sum_{i=1}^{n}w_{i}A_{i}\Big)\Phi^{-p}(P_{t}(\omega; {\mathbb A}))\Big\|\leq\frac{1}{4}\Big( \frac{K(M^{2}+m^{2})}{Mm} \Big)^{p}.
\end{align}
Since \eqref{22} is equivalent to \eqref{23}, so inequality \eqref{23} holds.
\end{proof}
\begin{remark}
If we put ${\mathbb A}=(A,B)$ and $\omega=(w_{1},w_{2})$ with $w_{1}=w_{2}=\frac{1}{2}$ in Theorem \ref{ref}, then we get \cite[Theorem 2.6]{zh} as follows:
\begin{align*}
\Phi^{2p}(\frac{A+B}{2})\leq \frac{(K(M^{2}+m^{2}))^{2p}}{16M^{2}m^{2}}\Phi^{2p}(A\sharp B).
\end{align*}
\end{remark}
\begin{theorem}\label{ref2}
Let ${\mathbb A}=(A_{1},\cdots,A_{n})$ be a $n$-tuple of positive definite matrices with $0<m\leq A_{i}\leq M,\ (i=1,\cdots,n)$ for some scalars $m\leq M$ and $\omega=(w_{1},\cdots,w_{n})$ a weight vector, and let $t\in [-1, 1]\backslash\{0\}$. Then for every positive unital linear map $\Phi$
\begin{align}\label{34}
P_{t}^{2p}(\omega;\Phi(\mathbb A))\leq\frac{(K(M^{2}+m^{2}))^{2p}}{16M^{2p}m^{2p}}\Phi^{2p}(P_{t}(\omega; {\mathbb A})),
\end{align}
where $p\geq2$ and $K=\frac{(M+m)^{2}}{4mM}$.
\end{theorem}
\begin{proof}
For $p\geq2$, we have
\begin{align*}
\Big\|P_{t}^{p}(\omega;\Phi(\mathbb A))&M^{p}m^{p}\Phi^{-p}(P_{t}(\omega; {\mathbb A}))\Big\|\\&\leq
\frac{1}{4}\Big\|P_{t}^{p}(\omega;\Phi(\mathbb A))+(M^{2}m^{2})^{\frac{p}{2}}\Phi^{-p}(P_{t}(\omega; {\mathbb A}))\Big\|^{2}\\&
\qquad  \qquad   \qquad \qquad \qquad \qquad  (\textrm{by Lemma \ref{6}(i)})\\&
\leq\frac{1}{4}\Big\|\Big(P_{t}^{2}(\omega;\Phi(\mathbb A))+M^{2}m^{2}\Phi^{-2}(P_{t}(\omega; {\mathbb A}))\Big)^{\frac{p}{2}}\Big\|^{2}\\&
\qquad  \qquad   \qquad \qquad \qquad \qquad (\textrm{by Lemma \ref{6}(ii)})\\&
=\frac{1}{4}\Big\|\Big(P_{t}^{2}(\omega;\Phi(\mathbb A))+M^{2}m^{2}\Phi^{-2}(P_{t}(\omega; {\mathbb A}))\Big)\Big\|^{p}\\&
\leq
\frac{1}{4}\Big\|\Big(P_{t}^{2}(\omega;\Phi(\mathbb A))+M^{2}m^{2}\Phi^{2}(P_{t}^{-1}(\omega; {\mathbb A}))\Big)\Big\|^{p}\\&
\qquad  \qquad   \qquad \qquad \qquad \qquad (\textrm{by \eqref{10}})\\&
\leq
\frac{1}{4}\Big\|K\Big(\sum_{i=1}^{n}w_{i}\Phi(A_{i})\Big)^{2}+M^{2}m^{2}K\Phi^{2}\Big(\sum_{i=1}^{n}w_{i}A_{i}^{-1}\Big)\Big\|^{p}\\&
\qquad  \qquad   \qquad \qquad \qquad \qquad (\textrm{by Lemma \ref{6}(iv)})\\&
=\frac{1}{4}\Big\|K\Phi^{2}\Big(\sum_{i=1}^{n}w_{i}A_{i}\Big)+M^{2}m^{2}K\Big(\sum_{i=1}^{n}w_{i}A_{i}^{-1}\Big)\Big\|^{p}\\&
=\frac{1}{4}(K(M^{2}+m^{2}))^{p}\qquad  \qquad   \qquad (\textrm{by \cite[4.7]{mlin}}).
\end{align*}
Therefore
\begin{align*}
\Big\|P_{t}^{p}(\omega;\Phi(\mathbb A))\Phi^{-p}(P_{t}(\omega; {\mathbb A}))\Big\|\leq\frac{1}{4}\Big( \frac{K(M^{2}+m^{2})}{Mm} \Big)^{p}.
\end{align*}
Since the last inequality is equivalent to \eqref{34}, thus this complets the proof.
\end{proof}
\begin{remark}
As special case for $\mathbb A=(A,B)$ and $\omega=(w_{1},w_{2})$ with $w_{1}=w_{2}=\frac{1}{2}$, Theorem \ref{ref2} is a refinement of Corollary \ref{30}.
\end{remark}

%==================================================================================================================================
\bibliographystyle{amsplain}

\begin{thebibliography}{99}

\bibitem{And} T. Ando, \textit{Concavity of certain maps on positive definite matrices and applications to
Hadamard products}, Linear Algebra Appl.
\textbf{27} (1979), 203–-241.

\bibitem{Ando} T. Ando and X. Zhan, \textit{Norm inequalities related to operator monotone functions}, Math. Ann.
\textbf{315} (1999), 771–-780.

\bibitem{Ba} M. Bakherad, \textit{Refinements of a reversed AM–-GM operator
inequality}, Linear and Multilinear Algebra \textbf{64}(9) (2016), 1687--1695.

\bibitem{Bha} R. Bhatia and F. Kittaneh, \textit{Notes on matrix arithmetic-geometric mean inequalities}, Linear Algebra
Appl. \textbf{308} (2000), 203–-211.

\bibitem{Ch} M. D. Choi, \textit{A Schwarz inequality for positive linear maps on C.-algebras}, Proc. Amer. Math. Soc, \textbf{8} (1974), 565-–574.

%\bibitem{Da}  C. Davis, \textit{ A Schwarz inequality for convex operator functions}, Illinois J. Math. \textbf{18} (1957), 42-–44.

 \bibitem{deh} M. Dehghani, M. Kian and Y. Seo, \textit{Matrix power means and the information monotonicity}, to appear in Linear Algebra Appl. (2017), http://dx.doi.org/10.1016/j.laa.2017.01.025.

\bibitem{fu} X. Fu and C. He, \textit{Some operator inequalities for positive linear maps}, Linear and Multilinear Algebra, \textbf{63}(3) (2015), 571--577.

\bibitem{Fuji}  M. Fujii, S. Izumino, R. Nakamato, and Y. Seo. \textit{Operator inequalities related to Cauchy-
Schwarz and H¨older-McCarthy inequalities}, Nihonkai Math. J., \textbf{8}(2): (1997), 117–-122.

\bibitem{Fu}  T. Furuta, J. Micic Hot, J. Pecaric and Y. Seo, \textit{ Mond–Pecaric method in operator inequal-
ities}, Zagreb, Element, (2005).

\bibitem{Ho} R. A. Horn and C. R. Johnson, \textit{Topics in Matrix Analysis,} Cambridge University Press, (1991).

\bibitem{Law} J. Lawson and Y. Lim, \textit{Karcher means and Karcher equations of positive definite operators}, Trans. Amer. Math. Soc., Series B,\textbf{1} (2014), 1–-22.

\bibitem{pal}Y. Lim and M. P´alfia, \textit{Matrix power means and the Karcher mean}, J. Funct. Anal. \textbf{262} (2012), 1498–-1514.

\bibitem{Yam} Y. Lim and T. Yamazaki, \textit{On some inequalities for the matrix power and Karcher means}, Linear Algebra Appl, \textbf{438} (2013), 1293–-1304.


\bibitem{lin} M. Lin, \textit{On an operator Kantorovich inequality for positive linear maps}, J. Math. Anal. Appl. \textbf{402} (2013), 127-–132.

\bibitem{mlin} M. Lin, \textit{Squaring a reverse AM-GM inequality}, Studia Math., \textbf{215} (2013), 189-–194.

\bibitem{Mar} A.W. Marshall and I. Olkin, \textit{Matrix versions of Cauchy and Kantorovich inequalities}, Aequationes
Math., \textbf{40} (1990), 89-–93.

\bibitem{zh} P. Zhang, \textit{More  operator inequalities for positive linear maps}, Banach J. Math. Anal. \textbf{9} (2015), no. 1, 166–-172.















\end{thebibliography}

\end{document}